\documentclass[12pt]{amsart}
\usepackage{amssymb,amsmath,graphics,verbatim}
\usepackage{latexsym}
\usepackage{eucal}

\usepackage{color}

\newtheorem{theorem}{Theorem}
\newtheorem{lemma}[theorem]{Lemma}
\newtheorem{prop}[theorem]{Proposition}

\theoremstyle{definition}
\newtheorem{definition}[theorem]{Definition}
\theoremstyle{remark}
\newtheorem{remark}[theorem]{Remark}

\newcommand{\bR}{\mathbb{R}}
\newcommand{\bC}{\mathbb{C}}
\newcommand{\cF}{\mathcal{F}}
\newcommand{\FP}{\mathrm{FP}}
\newcommand{\II}{\mathrm{II}}
\newcommand{\Vol}{\mathrm{Vol}}
\newcommand{\Volks}{\Vol_{\mathrm{KS}}}
\newcommand{\Tr}{\mathrm{Tr}}
\newcommand{\Volr}{\Vol_R}

\begin{document}
\title[Positivity of the renormalized volume of almost-Fuchsian 
manifolds]{Positivity of the renormalized volume of almost-Fuchsian hyperbolic $3$-manifolds}
\author{Corina Ciobotaru}
\thanks{C.~C. was supported by the FRIA}
\address{Corina Ciobotaru, Universit\'e de Gen\`eve, Section de math\'ematiques, 2-4 rue du Li\`{e}vre, CP 64, 1211 Gen\`eve 4, Switzerland}
\email{corina.ciobotaru@unige.ch}
\author{Sergiu Moroianu}
\thanks{S.~M. was partially supported by the CNCS project PN-II-RU-TE-2011-3-0053}
\address{Sergiu Moroianu, Institutul de Matematic\u{a} al Academiei Rom\^{a}ne\\ P.O. Box 1-764\\ RO-014700 Bucharest\\
Romania}
\email{moroianu@alum.mit.edu}
\date{\today}
\begin{abstract}
We prove that the renormalized volume of almost-Fuchsian hyperbolic $3$-ma\-ni\-folds 
is non-negative, with equality only for Fuchsian manifolds.
\end{abstract}
\maketitle

\section{Introduction}

The renormalized volume $\Volr$ is a numerical invariant associated to an infinite-volume Riemannian
manifold with some special structure near infinity, extracted 
from the divergent integral of the volume form. Early instances of renormalized volumes appear 
in Henningson--Skenderis~\cite{hs} for asymptotically hyperbolic Einstein metrics, 
and in Krasnov~\cite{Kr} for Schottky hyperbolic $3$-manifolds.
In Takhtajan--Teo~\cite{TaTe} the renormalized volume is identified to the so-called Liouville 
action functional, a cohomological quantity known since the pioneering work of 
Takhtajan--Zograf \cite{TaZo}
to be a K\"ahler potential 
for the Weil--Petersson symplectic form on the deformation space of certain Kleinian manifolds:
\begin{equation}
\partial\overline{\partial} \Volr= \frac{1}{8i}\omega_{\mathrm{WP}}.
\label{kp}\end{equation}

Krasnov--Schlenker~\cite{KS08} studied the renormalized volume using a geometric description 
in terms of foliations by equidistant surfaces.
In the context of quasi-Fuchsian hyperbolic $3$-manifolds they computed the Hessian of $\Volr$
at the Fuchsian locus. They also gave a direct proof of the identity \eqref{kp} 
in that setting.
Recently, Guillarmou--Moroianu~\cite{CS} studied the renormalized volume 
$\Volr$ in a general context, for geometrically finite hyperbolic $3$-manifolds without 
rank-$1$ cusps. There, $\Volr$ appears as the log-norm of a holomorphic section 
in the Chern--Simons line bundle over the Teichm\"uller space.

Huang--Wang~\cite{huangwang} looked at renormalized volumes in their study of 
almost-Fuchsian hyperbolic $3$-manifolds. However, their renormalization 
procedure does not involve uniformization of the surfaces at infinity, hence the invariant 
$RV$ thus obtained is constant (and negative) on the moduli space of almost-Fuchsian metrics.

There is a superficial analogy between $\Volr$ and the mass of asymptotically Euclidean 
mani\-folds. Like in the positive mass conjecture, one may ask if $\Volr$ is positive for all
convex co-compact hyperbolic $3$-manifolds, or at least for quasi-Fuchsian manifolds.
One piece of supporting evidence follows from the computation by Takhtajan--Teo~\cite{TaTe} 
of the variation of $\Volr$ (or equivalently, of the Liouville action functional) 
on deformation spaces. 
In the setting of quasi-Fuchsian manifolds, Krasnov--Schlenker~\cite{KS08} 
noted that the functional $\Volr$ vanishes at the Fuchsian locus. 
When one component of the boundary is kept fixed, the only critical point of $\Volr$ 
is at the unique Fuchsian metric. Moreover, this point is a local minimum because the 
Hessian of $\Volr$ is positive definite there as it coincides with the Weil-Petersson metric. 
Therefore, at least in a neighborhood of the Fuchsian locus, we do have positivity.
We emphasize that to ensure vanishing of the renormalized volume for Fuchsian manifolds, 
the renormalization procedure used in
Krasnov--Schlenker~\cite{KS08} differs from Guillarmou--Moroianu~\cite{CS} or from Huang--Wang~\cite{huangwang} by the universal constant $2\pi(1-g)$ where $g\geq 2$ is the genus.
It is the definition from Krasnov--Schlenker~\cite{KS08} that we use below. These results are
not sufficient to conclude that $\Volr$ is positive since the Teichm\"uller space 
is not compact and $\Volr$ is not proper (by combining the results in Schlenker~\cite{Sch} and Brock~\cite{Brock}, 
one sees that the difference between $\Volr$ and the Teichm\"uller distance is bounded, while 
the Teichm\"uller metric is incomplete).
Another piece of evidence towards positivity was recently found by Schlenker~\cite{Sch}, 
who proved that $\Volr$ is bounded from below by some explicit (negative) constant.

In this note we prove the positivity of $\Volr$ on the almost-Fuchsian space, 
which is an explicit open subset of the space of quasi-Fuchsian metrics. While this improves 
the local positivity result of Krasnov--Schlenker~\cite{KS08}, it does of course not prove
positivity for every quasi-Fuchsian metric, that is therefore left for further studies. 

\section{Almost-Fuchsian hyperbolic $3$-manifolds}

\begin{definition}
\label{def::quasi-Fuchsian}
A quasi-Fuchsian hyperbolic 3-manifold $X$ is the quotient of $\mathbb{H}^3$ by a 
quasi-Fuchsian group, i.e., a Kleinian group $\Gamma$ of $\mathrm{PSL}_2(\bC)$ 
whose limit set is a Jordan curve. 
\end{definition}
When the group $\Gamma$ is a co-compact Fuchsian group 
(a subgroup of $\mathrm{PSL}_2(\bR)$), the Jordan curve in question is the $1$-point 
compactification of the real line, and $\Gamma\backslash \mathbb{H}^3$ is called a Fuchsian 
hyperbolic 3-manifold.
Equivalently, a quasi-Fuchsian manifold $(X,g)$ is a complete hyperbolic 
$3$-manifold diffeomorphic to $\bR \times \Sigma_0$, where $\Sigma_0$ is a compact Riemann 
surface of genus $\geq 2$ and with the hyperbolic Riemannian metric $g$ on $X$ described 
as follows. There exist $t_{0}^{-} \leq t_{0}^{+} \in \bR$ such that the metric $g$ on 
$[t_{0}^+,\infty) \times \Sigma_0$, respectively
on $(-\infty,t_{0}^-] \times \Sigma_0$, is given by 
\begin{align}
\label{metricg}
g=dt^2+ g_t^{\pm},&& g_t^{\pm}=g_0^{\pm}((\cosh(t)+A^{\pm}\sinh(t))^2\cdot,\cdot),
\end{align}
where $t \in[t_{0}^+,\infty)$, respectively, $t \in (-\infty,t_{0}^-]$, $g_0^{\pm}$ is a metric on $\Sigma_{0}^{\pm} = \{t_{0}^{\pm}\} \times \Sigma_0$ and $A^{\pm}$ 
is a symmetric endomorphism of $T\Sigma_0^{\pm}$ satisfying the Gauss and Codazzi--Mainardi 
equations
\begin{align}
\det(A^{\pm})= {}&\kappa^{\pm}+1,\label{hGe}\\
 d^\nabla \II^{\pm} ={}&0.\nonumber
\end{align}

Here, $\kappa^{\pm}$ is the Gaussian curvature of $(\Sigma_0^{\pm},g_0^{\pm})$ and $d^\nabla$ 
represents the de Rham differential twisted by the Levi--Civita connection acting on $1$-forms 
with values in 
$T^*\Sigma_0^{\pm}$. By definition, $\II^{\pm}:=g_0^{\pm}(A^{\pm}\cdot,\cdot)$, 
called the second fundamental form 
of the embedding $\Sigma_0^{\pm}\hookrightarrow X$, is the bilinear form associated to $A^{\pm}$. 
Notice that the eigenvalues of $A^{\pm}$ should be less than $1$ in absolute value 
for the expression~(\ref{metricg}) 
to be a well-defined metric for all $t \in \bR$.

\begin{definition}[Uhlenbeck~\cite{Uhl}]
\label{def::almost-Fuch}
An almost-Fuchsian hyperbolic $3$-manifold $(X,g)$ is a quasi-Fuchsian hyperbolic $3$-manifold 
containing a closed minimal surface $\Sigma$ whose principal curvatures belong 
to $(-1,1)$.
\end{definition}

Roughly speaking, an almost-Fuchsian manifold is obtained as a small deformation 
of a Fuchsian manifold, which, by definition, is the quotient of $\mathbb{H}^3$ 
by the action of a co-compact Fuchsian group. In particular, Fuchsian manifolds 
are almost-Fuchsian. 

\begin{remark}
\label{rem::almost-F_all_t} By  Uhlenbeck~\cite[Theorem~3.3]{Uhl}, an almost-Fuchsian hyperbolic 
$3$-manifold $X$ 
admits a {\it unique} minimally embedded surface $\Sigma$, 
whose principal curvatures are thus in $(-1,1)$. 
By taking $\Sigma_0^{\pm}= \Sigma$, 
the expression~(\ref{metricg}) is well-defined for all $t \in \bR$. 
\end{remark}

\subsection{Funnel ends}
\label{subsec::funnel_ends}

Let $(X,g)=\Gamma\backslash \mathbb{H}^3$ be a quasi-Fuchsian manifold. 
Recall that the infinity of $X$ 
is defined as the space of geodesic rays escaping from every compact, modulo the equivalence 
relation of being asymptotically close to each other. By the Jordan separation theorem,
the complement of the limit set of $\Gamma$ consists of two disjoint topological disks. 
The infinity of a quasi-Fuchsian 
manifold is thus a disjoint union of two `ends', corresponding to geodesics in 
$\mathbb{H}^3$ pointing towards one or the other of these two connected components. 
For an end of $X$, a \emph{funnel} 
is a cylinder $[t_0,\infty)\times \Sigma \hookrightarrow X$ isometrically embedded in $X$ 
so that the pullback of the hyperbolic 
metric $g$ of $X$ is of the form~(\ref{metricg}) for $t \in [t_0,\infty)$, where $A$ satisfies 
the Gauss and Codazzi--Mainardi equations as above.  
Notice that the gradient of the function $t$ on the funnel $ [t_0,\infty)\times \Sigma$ 
is a geodesic vector field of length $1$; thus $\{\infty\} \times \Sigma$ is in bijection with
the corresponding end of $X$. A funnel has an obvious smooth compactification to a manifold 
with boundary, namely, $[t_0,\infty]\times \Sigma$. On this compactification, 
the Riemannian metric $e^{-2t}g$ is smooth in the variable $e^{-t}\in [0,e^{-t_0})$.
Define $h_0:=\lim\limits_{t\to\infty} e^{-2t}g$ to be the metric induced 
on the surface at infinity $\{\infty\}\times \Sigma$. Explicitly, 
\[h_0=\tfrac14 g_0((1+A)^2\cdot,\cdot).\] 
In this way, one obtains a smooth compactification $\overline{X}$ of $X$, together
with a metric at infinity, both depending at first sight on the funnels 
chosen inside each of the two ends of $X$. 

We emphasize however that each end of $X$ admits several funnel structures. 
Consider another cylinder $[t_0',\infty)\times\Sigma'\hookrightarrow X$ 
isometrically embedded in $X$, for a different function $t'$ 
with respect to which the metric 
$g$ takes the form~(\ref{metricg}). Then the gradient flow of $t'$ defines another 
foliation  $[t_0',\infty)\times\Sigma'$. If this funnel determines the same end of $X$ 
as $[t_0,\infty)\times \Sigma$, then the two funnels intersect near infinity.
Up to increasing $t_0'$ if necessary, we can assume that 
$[t_0',\infty)\times\Sigma' \hookrightarrow [t_0,\infty)\times \Sigma$. Moreover, 
for $t_0'$ large enough, the complement of the funnel $[t_0',\infty)\times\Sigma'$ in $X$
is geodesically convex, hence its boundary surface $\{t_0'\}\times \Sigma'$ 
intersects each half-geodesic along the $t$ 
flow in a unique point. Thus $\Sigma'$ is diffeomorphic to $\Sigma$.

The identity map of $X$ extends smoothly on the corresponding compactifications induced by the chosen foliation structures of each of the two ends; so the smooth compactification of $X$ is canonical (i.e., independent of the choice of the funnels). 
Moreover, the induced metrics $h_0,h_0'$ with respect to two foliations are conformal to each other. 
It follows that the metric $g$ induces a conformal class $[h_0]$ on 
$\{\infty\}\times\Sigma \subset \partial_\infty X$. 

Conversely, we recall that for a quasi-Fuchsian manifold $(X,g)$, 
every metric $h_0^\pm$ in the associated conformal class on each of the two ends of $X$ 
is realized 
(near infinity) by a unique funnel, using a special function $t$ that decomposes the funnel as 
presented above. 

\section{The renormalized volume}
\label{subsec::renorm_vol}
Let $(X,g)$ be a quasi-Fuchsian hyperbolic $3$-mani\-fold. 
For each of the two ends of $X$ we choose a funnel with foliation structure 
$[t_j,\infty)\times \Sigma_j$, where $j \in \{1,2\}$. Let $h_0^1,h_0^2$ be the 
corresponding metrics on the boundary 
at infinity of $X$. Choose $t_0:=\max\{t_1,t_2\}$ 
and set $\Sigma=\Sigma_1\sqcup\Sigma_2$, so 
that $[t_0,\infty)\times \Sigma$ is isometrically embedded in $X$. Denote by $h_0$ the  
metric $(h_0^1,h_0^2)$ on the disconnected surface $\Sigma$. 
For $t\geq t_0$, denote by $K_t$ the complement in $X$ of the funnels $[t,\infty)\times \Sigma$, 
which is a compact manifold with boundary $\{t\}\times \Sigma=:\Sigma_t$.
Let $\II^t$, $H^t:\Sigma_t \to \bR$ 
be the second fundamental form, respectively the mean curvature function 
of the boundary surfaces $\Sigma_t=\partial K_t$.

The renormalized volume of $X$ with respect to the metrics $h_0$ (or equivalently, with respect to
the corresponding functions $t$) is defined via the so-called Riesz regularization.
\begin{definition}
\label{def::renorm_vol}
Let $(X,g)$ be a quasi-Fuchsian hyperbolic $3$-manifold which is decomposed into a 
finite-volume open set $K$ and two funnels. 
As explained above, let $h_0$ be the metric in the induced conformal class at infinity of $X$ corresponding to $g$ and the chosen funnels. 
The renormalized volume with respect to $h_0$ is defined by
\[\Volr(X,g;h_0):=\Vol(K)+\FP_{z=0} \int_{X\setminus K} e^{-z|t|} dg,\]
where by $\FP$ we denote the finite part of a meromorphic function. 
\end{definition}
In Definition~\ref{def::renorm_vol}, we implicitly have used the fact, 
which follows from the proof of Proposition~\ref{propks} below, that the integral in the 
right-hand side is meromorphic in $z$. 
In Krasnov--Schlenker~\cite{KS08}, the renormalized volume is defined by integrating the 
volume form on increasingly large 
bounded domains and discarding some explicit terms which are divergent in the limit.
We refer, for example, to Albin~\cite{albin} for a discussion of the link between these two 
types of renormalizations. 
For the sake of completeness, we include here a proof of the equality between these 
two definitions. Some care is needed since the addition in the definition of an universal 
constant, harmless in Guillarmou--Moroianu~\cite{CS} or Huang--Wang~\cite{huangwang}, 
drastically alters the positivity properties
of $\Volr$.

\begin{prop}
\label{propks}
The quantity
\[
\Volks(X,g;h_0):=\Vol(K_t)-\tfrac{1}{4}\int_{\Sigma_t} H^t dg_t +t\pi \chi(\Sigma),
\]
called the (Krasnov--Schlenker) renormalized volume $\Volks$, is independent of $t\in[t_0,\infty)$, 
and coincides with the renormalized volume $\Volr(X,g;h_0)$. 
\end{prop}
The definition of $\Volks$ and the independence of $t$ are due to Krasnov--Schlenker~\cite{KS08},
see also Schlenker~\cite[Lemma 3.6]{Sch}.
\begin{proof}
We use the notation from the beginning of this section. Let
\begin{align*}
g=dt^2+ g_t,&& g_t=g_0((\cosh(t)+A\sinh(t))^2\cdot,\cdot)
\end{align*}
be the expression of the metric $g$ in the fixed product decomposition of the funnels $[t_0,\infty)\times \Sigma$, 
as in \eqref{metricg}. 
Recall that $\II^t=g_t(A_t \cdot, \cdot)=\frac{1}{2}g_t'$ and $H^t=\Tr(A_t)$. 
For every $t \in [t_0,\infty)$, one obtains 
\begin{align*}
 dg_t={}&[\cosh^2t+\det(A)\sinh^2(t)+ \Tr(A)\cosh(t)\sinh(t)]dg_0\\
={}&[\cosh^2t+(\kappa_{g_0}+1)\sinh^2(t)+ H^0\cosh(t)\sinh(t)]dg_0
\end{align*}
(in the second line we have used \eqref{hGe} and the definition of $H^0$), and
\begin{align*}
\tfrac{1}{2}g_t'={}&g_0((\cosh(t)+A\sinh(t))(\cosh(t)+A\sinh(t))'\cdot,\cdot)\\
={}&g_t((\cosh(t)+A\sinh(t))^{-2}(\cosh(t)+A\sinh(t))(\cosh(t)+A\sinh(t))'\cdot,\cdot)\\
={}&g_t((\cosh(t)+A\sinh(t))^{-1}(\sinh(t)+A\cosh(t))\cdot,\cdot),\\
\intertext{so}
A_t={}&(\cosh(t)+A\sinh(t))^{-1}(\sinh(t)+A\cosh(t)).
\end{align*}
Denote by $\lambda_1,\lambda_2$ the eigenvalues of the symmetric endomorphism
$A$, so $\lambda_1+\lambda_2=H^0$ and $\lambda_1\lambda_2 = \kappa_{g_0}+1$.
We deduce
\begin{align*}
H^tdg_t={}& (\cosh(2t)H^0+\sinh(2t)(\kappa_{g_0}+2))dg_0.
\end{align*}
The independence of $t$ is a straightforward consequence of the above formulas, we omit it since 
is coincides with Lemma 3.6 from Schlenker~\cite{Sch}. 
Let us prove the second part of the proposition. 
Fix $t \in (t_0, \infty)$ and use as a new variable $x \in [t,\infty)$. 
The above equations are of course valid if we replace $t$ by $x$ and $t_0$ by $t$. 
By the Gauss equation and the expressions of 
$dg_x$ and $H^xg_x$ with respect to $dg_t$ using the new variable $x$ we have:

\begin{align*}
\lefteqn{\FP_{z=0} \int_{X \setminus K_{t}} e^{-z|x|} dg}\\
&= \FP_{z=0} \int_{t}^{\infty} \int_{\Sigma_x}e^{-z|x|} dg_xdx\\
&=\FP_{z=0} \int_{t}^{\infty} \int_{\Sigma_t}e^{-z|x|}[(\cosh(x))^{2}+(\kappa_{g_t}+1)(\sinh(x))^2+ H^t\cosh(x)\sinh(x)]dg_tdx\\
&=\FP_{z=0} \int_{t}^{\infty} \int_{\Sigma_t}e^{-z|x|}\left(e^{2x}\left(\frac{\kappa_{g_t}}{4}
+ \frac{1}{2}+ \frac{H^{t}}{4}\right)+e^{-2x}\left(\frac{\kappa_{g_t}}{4}+ \frac{1}{2}- \frac{H^{t}}{4}\right) 
-\frac{\kappa_{g_t}}{2}\right)dg_tdx\\
&=\FP_{z=0}\left[\frac{1}{z-2}e^{(2-z)t}\int_{\Sigma_t}\left(\frac{\kappa_{g_t}}{4}+ \frac{1}{2}
+ \frac{H^{t}}{4}\right)dg_t\right.\\
&\hspace{1.7cm}
+ \left.\frac{1}{2+z}e^{-(z+2)t}\int_{\Sigma_t}\left(\frac{\kappa_{g_t}}{4}
+ \frac{1}{2}-\frac{H^{t}}{4}\right)dg_t+\frac{e^{zt}}{z}\int_{\Sigma_t}\frac{\kappa_{g_t}}{2} dg_t\right]\\
&=-\frac{1}{2}e^{2t}\int_{\Sigma_t}\left(\frac{\kappa_{g_t}}{4}+ \frac{1}{2}+ \frac{H^{t}}{4}\right)dg_t
+ \frac{1}{2}e^{-2t}\int_{\Sigma_t}\left(\frac{\kappa_{g_t}}{4}+ \frac{1}{2}-\frac{H^{t}}{4}\right)dg_t+t\pi \chi(\Sigma)\\
&=-\sinh(2t)\int_{\Sigma_t}\frac{\kappa_{g_t}+2}{4}dg_t -\cosh(2t)\int_{\Sigma_t}\frac{H^{t}}{4}dg_t+t\pi \chi(\Sigma)\\
&=-\frac{1}{4}\int_{\Sigma_t} H^t dg_t +t\pi \chi(\Sigma).
\end{align*}
\end{proof}

Given a quasi-Fuchsian manifold $(X,g)$, 
one would like to have a canonical definition of the renormalized volume, 
which does not depend on the additional choices of the metrics at infinity of $X$.

\begin{definition}
\label{def::renorm_vol_canonical}
The renormalized volume $\Volr(X,g)$ is defined as $\Volr(X,g;h_\cF)$, 
where the metrics $h_\cF$ at infinity of $X$ that are used for the renormalization 
procedure are the unique metrics in the conformal class $[h_0]$ having constant Gaussian curvature $-4$.
\end{definition}

This type of ``canonical'' renormalization first appeared in Krasnov \cite{Kr}.
Notice that, by the Gauss--Bonnet formula, the area, with respect to $h_\cF$, 
of the boundary at infinity $\{\infty\}\times\Sigma$ of each funnel of $X$ equals 
$-\frac{\pi\chi(\Sigma)}{2}$. The following lemma appears in Krasnov--Schlenker~\cite[Section 7]{KS08}; 
for the sake of completion we include below a (new) proof using our current definition of renormalized volume.

\begin{lemma}
\label{lem::maxvol}
Let $(X,g)$ be a quasi-Fuchsian hyperbolic $3$-manifold. 
Among all metrics $h_0\in[h_0]$ of area equal to $-\frac{\pi\chi(\Sigma)}{2}$, 
the renormalized volume $\Volr(X,g;h_0)$ attains its maximum for $h_0=h_\cF$.
\end{lemma}
The lemma holds evidently for every $\kappa<0$ when we maximize $\Volr$ among metrics of area
$-\frac{2\pi\chi(\Sigma)}{\kappa}$ in a fixed conformal class, the maximizer being the unique
metric with constant Gaussian curvature $\kappa<0$ in that conformal class.
\begin{proof}
From Guillarmou--Moroianu--Schlenker~\cite{GMS}, recall the conformal change formula of the renormalized volume. Let $h$ be a metric at infinity of $(X,g)$ and multiply $h$ by $e^{2\omega}$, for some smooth function $\omega:\{\infty\} \times \Sigma\to\bR$. We have that
\begin{align}\label{confchrenv}
\Volr(X,g;e^{2\omega}h)=\Volr(X,g;h)-\tfrac14 \int_\Sigma (|d\omega|^2_h+2\kappa_h\omega)dh.
\end{align}
In particular, for $h=h_\cF$ we obtain 
\begin{equation}\label{ineq1}
\Volr(X,g;e^{2\omega}h_\cF)-\Volr(X,g;h_\cF)\leq 2\int_\Sigma \omega dh_\cF.
\end{equation}
Now, we assume that $e^{2\omega}h_\cF$ has the same area as $h_\cF$; so 
$\int_\Sigma e^{2\omega}dh_\cF=\int_\Sigma dh_\cF$. Write $\omega=c+\omega^\perp$, 
with $c$ being a constant and $ \int_\Sigma \omega^\perp dh_\cF=0$ 
(this is Hodge decomposition for $0$-forms on $\Sigma$). 
Using the inequality $e^x\geq 1+x$, valid for all real numbers $x$, we get
\begin{align*}
\int_\Sigma dh_\cF={}&\int_\Sigma e^{2\omega} dh_\cF= e^{2c}\int_\Sigma e^{2\omega^\perp} dh_\cF \geq e^{2c}\int_\Sigma (1+2\omega^\perp)dh_\cF=e^{2c}\int_\Sigma dh_\cF,
\end{align*}implying that $c\leq 0$. Hence $\int_\Sigma \omega dh_\cF=\int_\Sigma c dh_\cF\leq 0$, proving the assertion of the lemma, in light of~(\ref{ineq1}).
\end{proof}

Moreover, when dilating $h_0$ by a constant greater than $1$, the renormalized volume increases. More precisely, we have:

\begin{lemma}
\label{lem::ineqdilat}
Let $(X,g)$ be a quasi-Fuchsian hyperbolic $3$-manifold. Let $c>0$ and 
let $[h_0]$ be the induced conformal class on the boundary at infinity of $X$, 
which by abuse of notation is denoted $\Sigma$. Let $h_0$ be a metric in $[h_0]$. Then
\[\Volr(X,g;c^2h_0)=\Volr(X,g;h_0)-\pi\chi(\Sigma) \ln c .
\]
\end{lemma}

\begin{proof}
This is a particular case of the formula~(\ref{confchrenv}) for the conformal change of the renormalized volume, in which $\omega$ is constant:
\[\Volr(X,g;e^{2\omega}h_0)=\Volr(X,g;h_0)-\tfrac14 \omega\int_\Sigma 2\kappa_{h_0}dh_0=\Volr(X,g;h_0)-\omega \pi \chi(\Sigma)
\](in the last equality we have used the Gauss--Bonnet formula).
\end{proof}

\section{Proof of the main result}
\label{subsec::renorm_vol_A-F}
Let $(X,g)$ be an almost-Fuchsian manifold. By Uhlenbeck~\cite{Uhl}, recall that $X$ 
contains a unique embedded minimal surface, which we denote $\Sigma$ in what follows. 
By considering the global decomposition $X=\bR \times \Sigma$ (see Remark~\ref{rem::almost-F_all_t}) 
we obtain two metrics $h_0^+,h_0^-$ in the corresponding conformal classes at $\pm \infty$ of $X$,
defined by $h_0^\pm:= (e^{-2|t|}g)_{|t=\pm\infty}.$ 
Using the Krasnov--Schlenker definition of the renormalized volume from Proposition~\ref{propks}, 
it is evident that, with respect to the globally defined function $t$ on the almost-Fuchsian
manifold $X$, we have
\[\Volr(X,g;h_0^\pm)=0.\]
This quantity is therefore not very interesting, but it will prove helpful when examining $\Volr(X,g)$. 
\begin{remark}
The vanishing of $\Volr(X,g;h_0^\pm)$ is essentially the content of Proposition~3.7 
in Huang--Wang~\cite{huangwang}, where a slightly different 
definition is used for the renormalized volume. In loc.\ cit.\ 
the renormalized volume $RV(X,g;h_0^\pm)$ equals
$\pi\chi(\Sigma)$ independently of the metric on $X$, and its sign is interpreted 
as some sort of ``negativity of the mass''. 
We defend here the view that the Krasnov--Schlenker definition seems to be the most meaningful, 
as opposed
to Guillarmou--Moroianu--Schlenker~\cite{GMS} or Huang--Wang~\cite{huangwang}, 
and that with this definition the sign of the volume appears to be positive, 
at least near the Fuchsian locus.
\end{remark}
Our goal is to control the renormalized volume of $(X,g)$ when the metric at $\pm \infty$ is $h_{\cF}^\pm$, 
the unique metrics of Gaussian curvature $-4$ inside the corresponding conformal class $[h_0^{\pm}]$ at infinity of $X$. 
Recall from Definition~\ref{def::renorm_vol_canonical} that, for this canonical choice (with non-standard constant $-4$) 
we obtain ``the'' renormalized volume of the almost-Fuchsian manifold $(X,g)$:
\[\Volr(X,g)=\Volr(X,g;h_\cF^\pm).\]

\begin{theorem}\label{thm::pos_renorm_vol}
The renormalized volume $\Volr(X,g)$ of an almost-Fuchsian hyperbolic $3$-ma\-ni\-fold $(X,g)$ is non-negative, 
being zero only at the Fuchsian locus, i.e., for $g$ as in Definition~\ref{def::quasi-Fuchsian} with $A=0$ and $g_0$ hyperbolic.
\end{theorem}
\begin{proof}
Denote the principal curvatures of the unique embedded minimal surface $\Sigma$ of $X$ by $\pm\lambda$ for some
continuous function $\lambda:\Sigma\to [0,\infty)$.  
Recall that $\sup\limits_{x \in \Sigma}|\lambda(x)| <1$ and that the decomposition of the metric $g$ takes the form~(\ref{metricg}), 
for all $t \in \bR$.

\begin{lemma}\label{lem::lemh0}
The Gaussian curvature of $h_0^\pm$ is bounded above by $-4$, with equality 
if and only if $X$ is Fuchsian. 
\end{lemma}

\begin{proof}
Let $\Sigma_t$ be the leaf of the foliation at time $t$. We compute the Gaussian curvature 
$\kappa_{h_0^{+}}$ as the limit of the curvature of $e^{-2 t}g_t$ as $t\to +\infty$. 
From the proof of Proposition~\ref{propks}, the shape operator of $\Sigma_t$ is 
$A_t=\tfrac12 g_t^{-1}{g}_t'
=(\cosh t+A\sinh t)^{-1}(A\cosh t +\sinh t)$.
By the Gauss equation~(\ref{hGe}), we get
\begin{align*}
\kappa_{g_t}={}&\det A_t-1\\
={}&\det[(1+A+e^{-2t}(1-A))^{-1}(1+A -e^{-2t}(1-A))]-1\\
={}&\frac{\left(1-e^{-2t}\frac{1-\lambda}{1+\lambda}\right)\left(1-e^{-2t}\frac{1+\lambda}{1-\lambda}\right)}
{\left(1+e^{-2t}\frac{1-\lambda}{1+\lambda}\right)\left(1+e^{-2t}\frac{1+\lambda}{1-\lambda}\right)}-1
\end{align*}
so $\kappa_{e^{-2t}g_t}=e^{2t}\kappa_{g_t}$
converges to $-2\left(\frac{1-\lambda}{1+\lambda}+\frac{1+\lambda}{1-\lambda} \right)\leq -4$
as $t\to\infty$. The inequality for $h_0^-$ is proved similarly. Clearly the equality 
holds if and only if $\lambda=0$, i.e., $A= 0$.
\end{proof}
Using the Gauss--Bonnet formula, Lemma~\ref{lem::lemh0} implies that the area of $(\{\pm\infty\}\times\Sigma, h_0^{\pm})$ is at most equal to $-\pi\chi(\Sigma)/2$, which, again by Gauss--Bonnet, is the area of $h_\cF^\pm$:
\begin{align*}
-4\int_\Sigma dh_0^\pm \geq \int_\Sigma \kappa_{h_0^\pm} dh_0^\pm =2\pi\chi(\Sigma)=-4 \int_\Sigma dh_\cF^\pm.
\end{align*} 
So
\begin{equation}\label{ineqvol}
\Vol(\Sigma, h_0^\pm)\leq \Vol(\Sigma, h_\cF^\pm)=-\pi\chi(\Sigma)/2, 
\end{equation}
with equality if and only if $\kappa_{h_0^\pm}=-4$, which is equivalent to $\lambda=0$. 

Let $c^2:=\Vol(\Sigma, h_\cF^\pm)/\Vol(\Sigma, h_0^\pm)$. By~(\ref{ineqvol}), $c\geq 1$. 
Applying Lemma~\ref{lem::ineqdilat} we obtain 
\[\Volr(X,g;h_0^\pm) \leq \Volr(X,g;c^2 h_0^\pm).\] 
Since, by definition $\Vol(\Sigma, c^2h_0^\pm)=
\Vol(\Sigma, h_\cF^\pm)$, Lemma~\ref{lem::maxvol} implies
\[\Volr(X,g;c^2h_0^\pm) \leq \Volr(X,g; h_\cF^\pm)=\Volr(X,g).\]
These inequalities are enough to conclude that $\Volr(X,g)\geq \Volr(X,g;h_0^\pm)=0$. 

Let us now analyze the equality case. If $(X,g)$ is Fuchsian, 
then the unique embedded minimal surface $\Sigma$ of $X$ has vanishing shape operator 
$A$, therefore $\lambda=-\lambda=0$. 
By the proof of Lemma~\ref{lem::lemh0}, we obtain that $\kappa_{h_0^{\pm}}=-4$, 
thus $\Volr(X,g)=0$. 

Conversely, assume that $(X,g)$ is almost-Fuchsian and that $\Volr(X,g)=0$. 
This implies that $\Volr(X,g;h_0^\pm)=\Volr(X,g;c^2 h_0^\pm)=\Volr(X,g)= \Volr(X,g;h_\cF^\pm)=0$, 
where $c^2$ was defined above as $\Vol(\Sigma, h_\cF^\pm)/\Vol(\Sigma, h_0^\pm)$. 
Thus, $c=1$ implying, by using the equality case in the inequality \eqref{ineqvol}, 
that $\kappa_{h_0^\pm} \equiv-4$
and that $\lambda=-\lambda=0$. Thus, the minimal surface $\Sigma$ 
is in fact totally geodesic, hence $(X,g)$ must be Fuchsian.
\end{proof}

\subsection*{Acknowledgments} We are indebted to Andy Sanders and Jean-Marc Schlenker, whom
we consulted about almost-Fuchsian metrics and renormalized volumes. 
We thank the anonymous referee for several remarks improving the presentation of the manuscript.
Colin Guillarmou pointed out to one of us
(S.~M.) why the question of positivity for the renormalized volume is still open; 
his explanations are kindly acknowledged.


\begin{thebibliography}{99}

\bibitem{albin} P.~Albin, {\sl Renormalizing curvature integrals on Poincar\'e-Einstein manifolds},
Adv.\ Math.\ \textbf{221} (2009), no.~1, 140--169.

\bibitem{Brock}
J.~Brock,
{\sl The Weil-Petersson metric and volumes of 3-dimensional hyperbolic convex cores,}
J.\ Amer.\ Math.\ Soc.\ {\bf 16} (2003), no.\ 3, 495--535 (electronic). 

\bibitem{CS}
C.~Guillarmou, S.~Moroianu, 
{\sl Chern-Simons line bundle on Teichm\"uller space,} Geometry \& Topology {\bf 18} (2014), 327--377.

\bibitem{GMS}
C.~Guillarmou, S.~Moroianu, J.-M.~Schlenker, 
{\sl The renormalized volume and uniformization of conformal structures}, preprint arXiv:1211.6705.

\bibitem{hs}
M.~Henningson, K.~Skenderis, 
{\sl The holographic Weyl anomaly}, J.\ High Energy Phys.\ (1998), no.~7, paper 23, 12 pp.\ (electronic). 

\bibitem{huangwang}
Zheng Huang, Biao Wang,
{\sl On almost-fuchsian manifolds,}
Trans.\ Amer.\ Math.\ Soc.\ {\bf 365} (2013), no.~9, 4679--4698. 

\bibitem{Kr} K.~Krasnov, {\sl Holography and Riemann surfaces},  
Adv.\ Theor.\ Math.\ Phys.\  \textbf{4} (2000) no.~4, 929--979. 

\bibitem{KS08}
K.~Krasnov, J.-M.~Schlenker, 
{\sl On the Renormalized Volume of Hyperbolic 3-Manifolds},
Commun Math.\ Phys.\ {\bf 279} (2008), no.~3, 637--668. 

\bibitem{Sch}
J.-M.~Schlenker, 
{\sl The renormalized volume and the volume of the convex core of quasifuchsian manifolds},
Math.\ Res.\ Lett.\ \textbf{20} (2013), 773--786.

\bibitem{TaTe} L.~A.~Takhtajan, L-P.~Teo, 
{\sl Liouville action and Weil-Petersson metric on deformation spaces, 
global Kleinian reciprocity and holography}, Comm.\ Math.\ Phys.\ \textbf{239} 
(2003) no.~1-2 , 183--240. 

\bibitem{TaZo} L.~A.~Takhtadzhyan, P.~G.~Zograf, 
{\sl On the uniformization of Riemann surfaces and on the Weil-Petersson metric 
on the Teichm\"uller  and Schottky spaces},
Mat. Sb. (N.S.) \textbf{132} (174) (1987), no.~3, 304--321, 444; 
translation in Math. USSR-Sb. 60 (1988), no.~2, 297--313.

\bibitem{Uhl}
K.~Uhlenbeck, 
{\sl Closed minimal surfaces in hyperbolic $3$-manifolds},
{Seminar on minimal submanifolds}, {Ann.\ Math.\ Stud.} {\bf 103},
Princeton Univ.\ Press, Princeton, NJ., 147--168, 1983.

\end{thebibliography}
\end{document}